\documentclass[12pt, reqno]{amsart}
\usepackage[left=1in,top=1in,right=1in,bottom=1in]{geometry}
\usepackage{amsmath, amssymb}
\usepackage{amsfonts}
\usepackage{mathrsfs}
\usepackage{bm}
\usepackage[arrow,matrix,curve,cmtip,ps]{xy}
\usepackage{graphicx}
\usepackage{subcaption}
\usepackage{comment}
\usepackage[hidelinks]{hyperref}
\usepackage[utf8]{inputenc}
\usepackage{tikz}
\usepackage{float}
\usepackage{amsthm}
\usepackage{enumitem}
\usepackage{subcaption}

\tikzset{node distance=2cm, auto}

\allowdisplaybreaks

\newtheorem{introthm}{Theorem}

\newtheorem{theorem}{Theorem}[section]
\newtheorem{lemma}[theorem]{Lemma}
\newtheorem{proposition}[theorem]{Proposition}

\newtheorem*{theorem*}{Theorem}
\theoremstyle{remark}

\newtheorem*{remark*}{Remark}

\newtheorem{definition}[theorem]{Definition}

\newtheorem{question}{Question}

\numberwithin{equation}{section}
\begin{document}


\title[Eigenvalues of graph Laplacian]{Flexibility of eigenvalues for graph Laplacians \\arising from genus 3 surfaces}

\author[A. Erchenko]{Alena Erchenko}
\address{Department of Mathematics, University of Oregon, Eugene, OR}
\email{\href{mailto:erchenko@uoregon.edu}{erchenko@uoregon.edu}}
\urladdr{\url{https://sites.google.com/view/aerchenko}}

\author[D. Jakobson]{Dmitry Jakobson}
\address{Department of Mathematics and Statistics, McGill University, Montr\'eal, Canada}
\email{\href{mailto:dmitry.jakobson@mcgill.ca}{dmitry.jakobson@mcgill.ca}}
\urladdr{\url{https://www.math.mcgill.ca/jakobson/}}

\author[A. Tsypin]{Allison Tsypin}
\address{Department of Mathematics, Dartmouth College,  Hanover, NH}
\email{\href{mailto:allison.l.tsypin.gr@dartmouth.edu}{allison.l.tsypin.gr@dartmouth.edu}}
\urladdr{\url{https://sites.google.com/view/atsypin/}}

\date{}

\date{}

\begin{abstract}
It is known that the small eigenvalues of the Laplacian of a Riemann surface close to the boundary of the modular space can be well approximated by the eigenvalues of the discrete Laplacian on a certain graph coming from the pair of pants decomposition of the surface. In this paper, we provide a complete description of the sets of eigenvalues of the weighted graph Laplacian for all graphs on four vertices that correspond to a valid pair of pants decomposition of a surface of genus 3.
\end{abstract}

\maketitle

\section{Introduction}\label{sec: introduction}

 Otal and Rosas \cite{OR} showed that on a 
hyperbolic closed surface of genus $g\geq 2$, the \linebreak$(2g-2)$-th Laplace eigenvalue $\lambda_{2g-2}\geq 1/4$. Thus, it is natural to study the first $(2g-3)$ eigenvalues and their multiplicities (see, for example, \cite{FB-P,FB-GM-P-P}). Moreover, it was recently shown (see \cite{AM1,AM2,HMT}) that for random surfaces of large 
genus, the first Laplace eigenvalue $\lambda_1$ approaches $1/4$. On the other hand, it is well-known that as Riemann surface degenerates, $\lambda_1$ approaches $0$. Motivated by the flexibility philosophy in dynamical systems formulated by A. Katok \cite{EK,BKRH}, it seems interesting to characterize 
the sets of $(2g-3)$ positive numbers in $[0,1/4]$ that can be realized as the first 
$(2g-3)$ eigenvalues of the hyperbolic Laplacian on some Riemann surface of genus 
$g$. In this paper, we discuss that question for degenerating Riemann surfaces. 
When Riemann surfaces degenerate, their eigenvalues 
are well approximated by the eigenvalues of the discrete Laplacian 
on certain graphs constructed using the separation of the surface 
into pairs of pants by small geodesics (see, for example, \cite{Colbois,Burger,Bat,SWY}). Motivated by this relation, we thus look at the discrete Laplacian on graphs. In this paper, we mainly focus on graphs on $4$ vertices which correspond to surfaces of genus $3$.

The only graphs on $4$ vertices that correspond to a valid decomposition of genus $3$ Riemann surface into pairs of pants are the star, the cycle, the path, the kite, and the complete graphs; they are depicted on FIG.~\ref{fig: genus 3 graphs}.

\begin{figure}[htb]
    \centering
    \begin{subfigure}[t]{0.16\textwidth}
        \centering
        \begin{tikzpicture}[node distance={9mm}, thick, main/.style = {draw, circle}] 
        \node[main] (1) {}; 
        \node[main] (2) [above of=1] {}; 
        \node[main] (3) [below right of=1] {}; 
        \node[main] (4) [below left of=1] {}; 
        \draw (1) -- (2); 
        \draw (1) -- (3); 
        \draw (1) -- (4);
        \end{tikzpicture} 
        \caption*{(A) Star}
    \end{subfigure}
    \hfill
    \begin{subfigure}[t]{0.16\textwidth}
        \centering
        \begin{tikzpicture}[node distance={15mm}, thick, main/.style = {draw, circle}] 
         \node[main] (1) {}; 
         \node[main] (2) [right of=1] {}; 
         \node[main] (3) [below of=2] {}; 
         \node[main] (4) [left of=3] {}; 
         \draw (1) -- (2); 
         \draw (2) -- (3); 
         \draw (3) -- (4);
         \draw (4) -- (1);
         \end{tikzpicture} 
         \caption*{(B) Cycle}
    \end{subfigure}
    \hfill
    \begin{subfigure}[t]{0.16\textwidth}
        \centering
        \begin{tikzpicture}[node distance={15mm}, thick, main/.style = {draw, circle}] 
        \node[main] (1) {}; 
        \node[main] (2) [right of=1] {}; 
        \node[main] (3) [below of=2] {}; 
        \node[main] (4) [left of=3] {}; 
        \draw (1) -- (2); 
        \draw (2) -- (3); 
        \draw (3) -- (4);
     \end{tikzpicture} 
        \caption*{(C) Path}
    \end{subfigure}
    \hfill
    \begin{subfigure}[t]{0.16\textwidth}
        \centering
        \begin{tikzpicture}[node distance={15mm}, thick, main/.style = {draw, circle}] 
         \node[main] (1) {}; 
         \node[main] (2) [right of=1] {}; 
         \node[main] (3) [below of=2] {}; 
         \node[main] (4) [left of=3] {}; 
         \draw (1) -- (2); 
         \draw (2) -- (3); 
         \draw (1) -- (3);
         \draw (3) -- (4);
         \end{tikzpicture} 
         \caption*{(D) Kite}
    \end{subfigure}
    \hfill
    \begin{subfigure}[t]{0.16\textwidth}
        \centering
        \begin{tikzpicture}[node distance={15mm}, thick, main/.style = {draw, circle}] 
        \node[main] (1) {}; 
        \node[main] (2) [right of=1] {}; 
        \node[main] (3) [below of=2] {}; 
        \node[main] (4) [left of=3] {}; 
        \draw (1) -- (2); 
        \draw (2) -- (3); 
        \draw (3) -- (4);
        \draw (4) -- (1);
        \draw (1) -- (3);
        \draw (2) -- (4);
        \end{tikzpicture} 
        \caption*{(E) $K_4$}
    \end{subfigure}
    \caption{Graphs on four vertices corresponding to genus 3 Riemann surfaces decomposing into pairs of pants}        
    \label{fig: genus 3 graphs}
\end{figure}

We recall that one of the eigenvalues of the graph Laplacian is necessarily $0$. We provide the full description of the set of the other eigenvalues for the graph Laplacian associated to graphs (A) - (E). See Section~\ref{sec: from surface to graph} for relevant definitions.

\begin{introthm}\label{thm: introthm}
The realizable eigenvalues $(x,y,z)$ of the graph Laplacian for the graphs (A) - (E) are the following.

\begin{enumerate}
\item For the star graph (A), $(x,y,z)$ are such that $x,y,z\geq 0$ and $$(6x^{2}-2xy-2xz+yz)(6y^2-2xy-2yz+xz)(6z^2-2xz-2yz+xy)\leq0.$$
\item For the cycle graph (B), $(x,y,z)$ are such that $x,y,z\geq 0$ and \[\max\{x,y,z\}\geq\frac{1}{2}(x+y+z).\]
\item For the path graph (C), $(x,y,z)$ are such that $x,y,z\geq 0$, $\max\{x,y,z\}\geq\frac{1}{2}(x+y+z)$, and 
\[
\Big(9(x^3+y^3+z^3)-13U+62xyz\Big)^2\leq (3S-2P)^2(9S-14P),
\]

where $S = x^2+y^2+z^2, \qquad P =xy+xz+yz,\qquad U = x^2(y+z) + y^2(x+z)+z^2(x+y)$.
\item For the kite graph (D), $(x,y,z)$ are such that $x,y,z\geq 0$ and at least one of the following inequalities holds 
\begin{itemize}
\item $\sqrt{3}|x-y|\geq x+y$,
\item $\sqrt{3}|x-z|\geq x+z$, or
\item $\sqrt{3}|y-z|\geq y+z$.
\end{itemize}
\item For $K_4$ graph (E), $(x,y,z)$ are such that $x,y,z\geq 0$.
\end{enumerate}
\end{introthm}

\begin{remark*}The result for the $K_4$ graph is well known, see 
\cite[\S4]{CdV88} or \cite[Thm. 2.5]{FGL}.   However, we provide 
the proof, motivated by the techniques in \cite{CdV88}, for completeness. It also leads us to formulate further questions relating to graph suspensions in Section~\ref{sec: open questions}.  

A problem of describing all attainable eigenvalues for weighted graph Laplacian on arbitrary 
graphs has been considered in \cite{FGL,CFGL}.  
A general description of the attainable eigenvalues on a $K_{1, k}$ graph (a generalized version of our star graph) is given in \cite[Thm. 2.6 and subsequent examples]{FGL}. In particular, they deduce an answer for our star graph (graph $K_{1,3}$) in Example 2.8 from the general expression, but the methods used are different from ours. Other graphs on 4 vertices have been considered in \cite[\S 4.3]{CFGL} where they find the potential boundaries of attainable eigenvalues. In our work, using methods independent of theirs, we show that the actual boundaries for the realizable eigenvalues for the weighted graph Laplacian on 4 vertices have the same shape as their potential boundaries.

A complete graph $K_4$ without an edge is the only connected graph on $4$ vertices not covered by Theorem~\ref{thm: introthm}. We expect that the methods developed in the paper should allow to obtain the full description of eigenvalues in that case. However, we omit this case as we are focused on the graphs coming from the decomposition of genus $3$ Riemann surface into pairs of pants.
\end{remark*}
Note that we allow zero edge weights, corresponding to  ``deleting" the relevant edge. Moreover, we do not order eigenvalues so we consider the $n$-tuples of eigenvalues. Also, we observe that if all edge weights are multiplied by a positive constant, the whole Laplacian matrix will be multiplied by the same constant, and thus all eigenvalues will be scaled by this constant. Accordingly, we can always scale the set of eigenvalues by a positive constant $t$: the set $\{\lambda_0, \lambda_1, \lambda_2, \ldots, \lambda_n\}$ can be attained on a given graph if and only if the set $\{t\lambda_0, t\lambda_1, t\lambda_2, t\ldots, t\lambda_n\}$ can be attained on that graph. 
This allows us to apply a normalization condition: since we can scale at will, we can impose a fixed sum of the eigenvalues. For graphs with four vertices, we will sometimes impose that the eigenvalues must sum to 8. In particular, the sets of the realizable eigenvalues with normalization $x+y+z=c>0$ for the graphs (A) - (E) are shown on FIG.~\ref{tikz: eigenvalues} in gray (positive eigenvalues) and black (at least one of the eigenvalues is $0$).

\bigskip
\begin{figure}[htb]
    \centering
\begin{tikzpicture}[scale=1.5]
\filldraw[fill=gray!50,draw=gray!50] plot[smooth, tension=0.3] coordinates {
    (-1,0) (-0.5,0) (-0.45, 0.289) (-0.48, 0.243) (-0.6, 0.15)
    (-0.79,0.07) (-1,0)};
     \begin{scope}[xscale=-1]
\filldraw[fill=gray!50,draw=gray!50] plot[smooth, tension=0.3] coordinates {
    (-1,0) (-0.5,0) (-0.45, 0.289) (-0.48, 0.243) (-0.6, 0.15)
    (-0.79,0.07) (-1,0)};
     \end{scope}

\begin{scope}[rotate around={120:(0,0.577)}]
\filldraw[fill=gray!50,draw=gray!50] plot[smooth, tension=0.3] coordinates {
    (-1,0) (-0.5,0) (-0.45, 0.289) (-0.48, 0.243) (-0.6, 0.15)
    (-0.79,0.07) (-1,0)};
     \begin{scope}[xscale=-1]
\filldraw[fill=gray!50,draw=gray!50] plot[smooth, tension=0.3] coordinates {
    (-1,0) (-0.5,0) (-0.45, 0.289) (-0.48, 0.243) (-0.6, 0.15)
    (-0.79,0.07) (-1,0)};
     \end{scope}
\end{scope}
\begin{scope}[rotate around={240:(0,0.577)}]
\filldraw[fill=gray!50,draw=gray!50] plot[smooth, tension=0.3] coordinates {
    (-1,0) (-0.5,0) (-0.45, 0.289) (-0.48, 0.243) (-0.6, 0.15)
    (-0.79,0.07) (-1,0)};
     \begin{scope}[xscale=-1]
\filldraw[fill=gray!50,draw=gray!50] plot[smooth, tension=0.3] coordinates {
    (-1,0) (-0.5,0) (-0.45, 0.289) (-0.48, 0.243) (-0.6, 0.15)
    (-0.79,0.07) (-1,0)};
     \end{scope}
\end{scope}
\draw[draw=black, very thick] 
     (-1,0) -- (1,0) -- (0,1.732) -- cycle;
\node at (0,-1) {(A) Eigenvalues of the star graph};
\end{tikzpicture}
\hspace{2cm}
\begin{tikzpicture}[scale=1.5]
\filldraw[fill=gray!50,draw=black, thick] 
    (-1,0) -- (1,0) -- (0,1.732) -- cycle;
\filldraw[fill=white,draw=gray!50, ultra thin] 
    (0,0) -- (-0.5,0.866) -- (0.5,0.866) -- cycle;
\node at (0,-1) {(B) Eigenvalues of the cycle graph};
\end{tikzpicture}

\begin{tikzpicture}[scale=1.5]
    \filldraw[fill=gray!50,draw=gray!50] plot[smooth, tension=0.3] coordinates {
    (-1,0) (-0.79,0.07) (-0.68,0.1) (-0.46,0.16)  (-0.22,0.166) (-0.128,0.135) (-0.086,0.1065)(0,0) (-1,0)};
    \begin{scope}[xscale=-1]
   \filldraw[fill=gray!50,draw=gray!50] plot[smooth, tension=0.3] coordinates {
    (-1,0) (-0.79,0.07) (-0.68,0.1) (-0.46,0.16)  (-0.22,0.166) (-0.128,0.135) (-0.086,0.1065)(0,0) (-1,0)};
  \end{scope}
\begin{scope}[rotate around={120:(0,0.577)}]
\filldraw[fill=gray!50,draw=gray!50] plot[smooth, tension=0.3] coordinates {
    (-1,0) (-0.79,0.07) (-0.68,0.1) (-0.46,0.16)  (-0.22,0.166) (-0.128,0.135) (-0.086,0.1065)(0,0) (-1,0)};
    \begin{scope}[xscale=-1]
   \filldraw[fill=gray!50,draw=gray!50] plot[smooth, tension=0.3] coordinates {
    (-1,0) (-0.79,0.07) (-0.68,0.1) (-0.46,0.16)  (-0.22,0.166) (-0.128,0.135) (-0.086,0.1065)(0,0) (-1,0)};
  \end{scope}
\end{scope}

  \begin{scope}[rotate around={240:(0,0.577)}]
\filldraw[fill=gray!50,draw=gray!50] plot[smooth, tension=0.3] coordinates {
    (-1,0) (-0.79,0.07) (-0.68,0.1) (-0.46,0.16)  (-0.22,0.166) (-0.128,0.135) (-0.086,0.1065)(0,0) (-1,0)};
    \begin{scope}[xscale=-1]
   \filldraw[fill=gray!50,draw=gray!50] plot[smooth, tension=0.3] coordinates {
    (-1,0) (-0.79,0.07) (-0.68,0.1) (-0.46,0.16)  (-0.22,0.166) (-0.128,0.135) (-0.086,0.1065)(0,0) (-1,0)};
  \end{scope}
\end{scope}
     \draw[draw=black, very thick] 
    (-1,0) -- (1,0) -- (0,1.732) -- cycle;
\node at (0,-1) {(C) Eigenvalues of the path graph};
\end{tikzpicture}
\hspace{2cm}
\begin{tikzpicture}[scale=1.5]
\filldraw[fill=gray!50, draw=gray!, ultra thin] 
    (-1,0) -- (1,0) -- (0.789,0.366) -- cycle;
    \begin{scope}[xscale=-1]
   \filldraw[fill=gray!50, draw=gray!50, ultra thin] 
    (-1,0) -- (1,0) -- (0.789,0.366) -- cycle;
  \end{scope}

  \begin{scope}[rotate around={120:(0,0.577)}]
       \filldraw[fill=gray!50, draw=gray!50, ultra thin] 
    (-1,0) -- (1,0) -- (0.789,0.366) -- cycle;
    \begin{scope}[xscale=-1]
   \filldraw[fill=gray!50, draw=gray!50, ultra thin] 
    (-1,0) -- (1,0) -- (0.789,0.366) -- cycle;
    \end{scope}
\end{scope}

  \begin{scope}[rotate around={240:(0,0.577)}]
       \filldraw[fill=gray!50, draw=gray!50, ultra thin] 
    (-1,0) -- (1,0) -- (0.789,0.366) -- cycle;
    \begin{scope}[xscale=-1]
   \filldraw[fill=gray!50, draw=gray!50, ultra thin] 
    (-1,0) -- (1,0) -- (0.789,0.366) -- cycle;
    \end{scope}
\end{scope}
\draw[draw=black, very thick] 
    (-1,0) -- (1,0) -- (0,1.732) -- cycle;
\node at (0,-1) {(D) Eigenvalues of the kite graph};
\end{tikzpicture}

\begin{center}
\begin{tikzpicture}[scale=1.5]
\filldraw[fill=gray!50, draw=black, very thick] 
    (-1,0) -- (1,0) -- (0,1.732) -- cycle;
\node at (0,-1) {(E) Eigenvalues of the $K_4$ graph};
\end{tikzpicture}
\end{center}
\caption{Realizable eigenvalues for the graphs (A) - (E) for the normalization that the sum of eigenvalues is a fixed constant $>0$ (shown in gray)}
\label{tikz: eigenvalues}
\end{figure}

\subsection*{Organization of the paper} In Section~\ref{sec: from surface to graph}, we recall the construction of associate graph to a hyperbolic closed surface and the relations of the corresponding Laplacians. In the further sections we prove the statement of Theorem~\ref{thm: introthm}. More precisely, we consider the star graph in Section~\ref{sec: star graph}, the cycle graph in Section~\ref{sec: cycle graph}, the path graph in Section~\ref{sec: path graph}, the kite graph in Section~\ref{sec: kite graph}, and, finally, the complete graph in Section~\ref{sec: complete graph}. In Section~\ref{sec: open questions}, we formulate related open questions that might be of independent interest.

\section{From a surface to a graph}\label{sec: from surface to graph} 

Consider a hyperbolic surface $S$ of genus $g\geq 2$. We say $\Lambda\subset S$ is a partition of $S$ if it is a union of simple closed pairwise non-intersecting geodesics. We associate an unoriented graph $\Gamma =(V,E)$ to $\Lambda$ in the following way. The set of vertices $V$ is the set of connected components of $S\setminus\Lambda$ with $n+1$ elements. The set of edges $E$ is obtained from the geodesics in $\Lambda$. To be more precise, two vertices in $V$ are connected by an edge if the connected components of $S\setminus \Lambda$ corresponding to them are joined by a geodesic in $\Lambda$. Let $v_i$ be the area of $V_i$. To each edge between vertices $V_i$ and $V_j$, we associate a weight which is equal to the sum of lengths of geodesics connecting the pieces associated to $V_i$ and $V_j$. We denote by $i(p)$, $j(p)$ the indices of the end points of the edge $e_p$ in $\Gamma$, where $1\leq p\leq r$. On the space $V^*\cong\mathbb R^{n+1}$ of function on $\Gamma$ with the scalar product $<F,G>=\sum\limits_{i=1}^{n+1}v_iF_iG_i$, we define a graph Laplacian (a combinatorial Laplacian) as the quadratic form $Q(F)=\frac{1}{\pi}\sum\limits_{p=1}^r l(e_p)(F_{j(p)}-F_{i(p)})^2$, where $l(e_p)$ is the weight of the edge $e_p$ for all $p$. The eigenvalues of the graph Laplacian are the eigenvalues of the associated matrix of the quadratic form.

In this paper, we concentrate on partitions $\Lambda$ so that each connected component of $S\setminus\Lambda$ is either a pair of pants or a pair of pants with trouser legs connected. Such a decomposition has $2g-2$ connected components. Thus, for genus $3$ surfaces we obtain graphs on $4$ vertices of form (A) - (E) as in Section~\ref{sec: introduction}.

The eigenvalues of the graph Laplacian associated to the decomposition above into pair of pants are closely related to the eigenvalues of the Laplacian of the initial hyperbolic metric. To be more precise, Colbois showed the following \cite{Colbois}. Fix a closed surface $S$ of genus $g\geq 2$. Consider a family of hyperbolic metrics $\{g_t\}_{0<t\leq 1}$ on $S$. Let $\Lambda$ be the partition of $(S,g_1)$ into pair of pants as above. Assume that the metrics $g_t$ are such that the $g_t$-lengths of the geodesics (for $g_t$ and $g_1$) in $\Lambda$ are $t$ multiples of $g_1$-lengths of those geodesics (we scale the length of geodesics but preserve the twist parameter). Let $0<\sigma_1(t)\leq\ldots\leq\sigma_{2g-3}(t)$ be the first $2g-3$ nonzero eigenvalues of the Laplacian associated to $g_t$. Denote by $0<\lambda_1\leq\ldots\leq\lambda_{2g-3}$ be the eigenvalues of the graph Laplacian associated to $(S,g_1)$ with the partition $\Lambda$. By Th\'eor\`eme in \cite{Colbois},
\[\lim_{t\rightarrow 0}\frac{\sigma_k(t)}{t}=\lambda_k\qquad\text{for}\quad 1\leq k\leq 2g-3.\]

\section{Star graph}\label{sec: star graph}
The star graph is a four-vertex graph with one vertex having degree three and all others having degree one (i.e., the bipartite graph $K_{1,3}$). In this section we show that the eigenvalue equations are symmetric, allowing us to express the realizable eigenvalues of the star graph as the solutions to a cubic equation.

\begin{theorem}(compare with \cite[Example 2.8]{FGL})\label{theorem_star}
    The triple of eigenvalues attained by the star graph on four vertices are precisely $(x,y,z)$ satisfying $$(6x^{2}-2xy-2xz+yz)(6y^2-2xy-2yz+xz)(6z^2-2xz-2yz+xy)\leq0,\quad \text{where}\quad x,y,z\geq 0.$$
\end{theorem}

\begin{proof}
Consider a star graph where vertices $v_1, v_2, v_3$ each have a single attached edge of weight $l_1, l_2, l_3$, respectively, with the other end being $v_4$. Then, the weighted graph Laplacian has form

\begin{align*}\label{star_matrix}
\begin{bmatrix} 
l_1 & 0 & 0 & -l_1 \\
0 & l_2 & 0 & -l_2 \\
0 & 0 & l_3 & -l_3 \\
-l_1 & -l_2 & -l_3 & l_1+l_2+l_3
\end{bmatrix}.
\end{align*}

The characteristic equation of this matrix is 
\begin{equation*}
\lambda^4-2(l_1+l_2+l_3)\lambda^3+3(l_1l_2+l_1l_3+l_2l_3)\lambda^2-4l_1l_2l_3\lambda=0.
\end{equation*}

Therefore, $\lambda_0=0$ is an eigenvalue and the other eigenvalues $\lambda_1, \lambda_2$, and $\lambda_3$ satisfy the equation
\begin{equation}\label{star_equat}
\lambda^3-2(l_1+l_2+l_3)\lambda^2+3(l_1l_2+l_1l_3+l_2l_3)\lambda-4l_1l_2l_3=0.
\end{equation}

Thus, we have
\begin{align}
&\lambda_1+\lambda_2+\lambda_3 = 2(l_1+l_2+l_3),\nonumber\\
&\lambda_1\lambda_2+\lambda_1\lambda_3+\lambda_2\lambda_3 = 3(l_1l_2+l_1l_3+l_2l_3),\label{star_system}\\
&\lambda_1\lambda_2\lambda_3 = 4l_1l_2l_3.\nonumber
\end{align}

These equalities give us three symmetric polynomials in terms of $l_1, l_2, l_3$, so we can rewrite them as a cubic equation whose solutions are $l_1, l_2, l_3$:

\begin{align}
    L^{3}-\frac{1}{2}\left(\lambda_1+\lambda_2+\lambda_3\right)L^{2}+\frac{1}{3}\left(\lambda_1\lambda_2+\lambda_1\lambda_3+\lambda_2\lambda_3\right)L-\frac{1}{4}\left(\lambda_1\lambda_2\lambda_3\right)=0.\label{star_length_cubic}
\end{align}

Then, attainable eigenvalues exactly correspond to the case where this equation has three real roots, i.e., when the discriminant of the cubic equation is nonnegative, with the additional constraint that all $l_i$ and $\lambda_i$, $i=1,2,3$, are nonnegative. 

The cubic equation has the form $AL^3+BL^2+CL+D=0$ with 
\begin{alignat*}{2}
   &A = 1,\qquad &&B = -\frac{1}{2}\left(\lambda_1+\lambda_2+\lambda_3\right),\\ &C = \frac{1}{3}\left(\lambda_1\lambda_2+\lambda_1\lambda_3+\lambda_2\lambda_3\right),\qquad &&D = -\frac{1}{4}\left(\lambda_1\lambda_2\lambda_3\right). 
\end{alignat*}
Using the formula for the discriminant of a cubic equation, namely, \[18ABCD-4B^3D+B^2C^2-4AC^3-27A^2D^2,\] and renaming $(\lambda_1, \lambda_2, \lambda_3)=(x,y,z),$ we get the constraint that 
\begin{align*}
    (6x^{2}-2xy-2xz+yz)(6y^2-2xy-2yz+xz)(6z^2-2xz-2yz+xy)\leq0.
\end{align*}

Since the eigenvalues of the Laplacian are always nonnegative, we also have constraints $x,y,z\geq0$. Finally, we note that if all $\lambda_i$, $i=1,2,3$, are nonnegative, plugging in a value $L<0$ to the cubic equation \eqref{star_length_cubic} will give a strictly negative result. As a result, the roots $l_i$, $i=1,2,3$, must be nonnegative, and so negative weights do not arise. Hence, the eigenvalues of the star graph are exactly those described in Theorem \ref{theorem_star}.

\end{proof}

\section{Cycle}\label{sec: cycle graph}
Consider the cycle graph on four vertices, i.e., let $v_1,v_2,v_3,v_4$ be the vertices, then the set of edges consists of $v_4v_1$ and $v_iv_{i+1}$, $i=1,2,3$. To describe the realizable eigenvalues on the cycle graph, we introduce a convenient parametrization and express the eigenvalues in terms of these parameters.

\begin{theorem}\label{theorem_cycle}
    The triple of eigenvalues attained by the cycle graph on four vertices are precisely $(x,y,z)$ where $x,y,z\geq0$ and $\max\{x,y,z\}\geq \frac{1}{2}(x+y+z)$.
\end{theorem}
\begin{proof}
We show the triple $(\lambda_1, \lambda_2, \lambda_3)$ of eigenvalues (we exclude $\lambda_0=0$) with $\lambda_1+\lambda_2+\lambda_3=8$ and $\lambda_j\geq 0$ for all $j=1,2,3$ is realizable by the cycle graph if and only if $\lambda_i\geq 4$ for some $i\in\{1,2,3\}$.

    Let $\{l_1, l_2, l_3, l_4\}$ be the weights of edges in the considered cycle graph on four vertices. Then, the weighted graph Laplacian has form

    \begin{align*}
    \begin{bmatrix}
    l_1+l_4 & -l_1 & 0 & -l_4 \\
    -l_1 & l_1+l_2 & -l_2 & 0 \\
    0 & -l_2 & l_2+l_3 & -l_3 \\
    -l_4 & 0 & -l_3 & l_3+l_4 \\
    \end{bmatrix}.
    \end{align*}

Calculating the characteristic equation, we get that, excluding the necessarily $0$ eigenvalue, the rest of eigenvalues must satisfy
\begin{align*}
    \lambda^3 - 2 \lambda^2 (l_1 + l_2 + l_3 + l_4) + \lambda (3(l_1 l_2 + l_2 l_3 + l_3 l_4 + l_4 l_1) &+ 4(l_1 l_3 + l_2 l_4))\\ &- 4 (l_1 l_2 l_3 + l_1 l_2 l_4 + l_1 l_3 l_4 + l_2 l_3 l_4) = 0.
\end{align*}
Thus, the roots $\lambda_1, \lambda_2, \lambda_3$ of this cubic equation satisfy the following system of equations.

\begin{align}
    \lambda_1+\lambda_2+\lambda_3 & = 2(l_1+l_2+l_3+l_4), \nonumber\\
    \lambda_1\lambda_2 + \lambda_2\lambda_3 + \lambda_1\lambda_3 & = 3 (l_1 l_2 + l_2 l_3 + l_3 l_4 + l_4 l_1) + 4 (l_1 l_3 + l_2 l_4), \label{eq: cycle eigenvalues}\\
    \lambda_1\lambda_2\lambda_3 & = 4 (l_1 l_2 l_3 + l_1 l_2 l_4 + l_1 l_3 l_4 + l_2 l_3 l_4).\nonumber
\end{align}

Let 
\begin{equation}\label{cycle: parameters vs lengths}
l_1+l_3=a,  \hspace{10pt}l_2+l_4=b,\hspace{10pt}l_1l_3=c,  \hspace{10pt}l_2l_4=d.
\end{equation}
We want to express our constraints on the roots using $a,b,c$, and $d$. Note that the new parameters are symmetric polynomials in terms of the weights. 
In order to make sure we can recover positive real weights $l_1, l_2, l_3, l_4$ from $a,b,c,d$, it is necessary and sufficient to have the following inequalities hold.
\begin{equation}\label{cycle: constraints}
a,b,c,d\geq 0, \hspace{10pt}4c\leq a^2,  \hspace{10pt}4d\leq b^2.
\end{equation}
Our normalization $\lambda_1+\lambda_2+\lambda_3=8$ translates into $l_1+l_2+l_3+l_4=4$, i.e., $b=4-a$. As a result, the system of equations \eqref{eq: cycle eigenvalues} can be written in the following form.

\begin{align}
    \lambda_1+\lambda_2+\lambda_3 & = 2(a+b)=8, \nonumber\\
    \lambda_1\lambda_2 + \lambda_2\lambda_3 + \lambda_1\lambda_3 &=3a(4-a)+4(c+d),\label{cycle: eigenvalues vs parameters}\\
    \lambda_1\lambda_2\lambda_3 & = 4(ad+(4-a)c).\nonumber
\end{align}

Let $\Lambda_1=\Lambda$ be one of of our eigenvalues. In particular, $\Lambda$ is a root of the equation \begin{align}
    p(\lambda):=\lambda^3-8\lambda^2+(3a(4-a)+4(c+d))\lambda-4(ad+(4-a)c) = 0. \label{cycle: Lambda_cubic}
\end{align}

Thus, factoring out $\lambda-\Lambda$ in the polynomial above, we obtain that the other roots of \eqref{cycle: Lambda_cubic} must be the zeros of the quadratic polynomial $\lambda^2+(\Lambda-8)\lambda+3a(4-a)+4(c+d)+(\Lambda-8)\Lambda$. Therefore, the other two eigenvalues of the considered graph Laplacian are 
\begin{align}
    \Lambda_{2,3} & = \frac{1}{2}\Bigg ((8-\Lambda)\pm\sqrt{(\Lambda-8)^2-4(3a(4-a)+4c+4d+\Lambda(\Lambda-8))}\Bigg ).\label{cycle: L2,3 expression} 
\end{align}

\begin{lemma}\label{lemma: cycle all smaller than 4}
     If $0\leq \Lambda_i<4$ for $i=1,2,3$ and $\Lambda_1+\Lambda_2+\Lambda_3=8$, the triple ($\Lambda_1, \Lambda_2, \Lambda_3)$ is not attainable.
 \end{lemma}
\begin{proof}
$\Lambda_1,\Lambda_2,\Lambda_3$ are roots of the polynomial $p(\lambda)$ of degree $3$ defined in \eqref{cycle: Lambda_cubic}. We have
  \begin{equation}\label{eq: p(4)}
p(4)=-(a(a^2-4c)+(4-a)((4-a)^2-4d))\leq 0 \qquad\text{by \eqref{cycle: constraints}}.
\end{equation}  

Moreover, 
\begin{equation}\label{eq: p(8)}
p(8)=4(6a(4-a)+(4+a)c+(8-a)d)\geq 0\qquad\text{by \eqref{cycle: constraints}.}
\end{equation}

Combining \eqref{eq: p(4)} and \eqref{eq: p(8)}, we obtain that $p(\lambda)$ has a root in $[4,8]$ so at least one of the eigenvalues should be $\geq 4$.
\end{proof}

 \begin{lemma}\label{lemma: cycle one larger than 4}
     If $\Lambda_1,\Lambda_2,\Lambda_3\geq 0$, $\Lambda_1+\Lambda_2+\Lambda_3=8$, and $\Lambda_i\geq 4$ for some $i\in\{1,2,3\}$, then the triple $(\Lambda_1,\Lambda_2,\Lambda_3)$ is attainable.
 \end{lemma}
 \begin{proof}
Consider a given triple $(\Lambda_1, \Lambda_2, \Lambda_3)$ of non-negative numbers summing to 8 and with one of the three eigenvalues greater than or equal to 4. Without loss of generality, we can assume that $\Lambda_1=4+s\geq 4$, i.e., $0\leq s\leq 4$. Denote $\Lambda_2=t$, then $\Lambda_3=4-s-t$ where $0\leq t$ and $s+t\leq 4$. 

By \eqref{cycle: Lambda_cubic}, we want to show that there exist $a, c,$ and $d$ such that
\begin{align}
3a(4-a)+4(c+d)&=16-s^2+t(4-s)-t^2,\nonumber\\
4(ad+(4-a)c)&=(16-s^2)t-(4+s)t^2,\nonumber\\
0\leq &a\leq4,\label{cycle: inverse equation}\\ 0\leq &c\leq\frac{a^2}{4},\nonumber\\ 0\leq &d\leq\frac{(4-a)^2}{4}.\nonumber
\end{align}

We note that we have $a=2$ is a solution only if $2(16-s^2+t(4-s)-t^2-12)=(16-s^2)t-(4+s)t^2$. The previous equation is equivalent to $(s+2)t^2+(s^2-2s-8)t+(8-2s^2)=0$ which in turns factors as $(s+t-2)(s+2)(t-2)=0$. Since we assumed $0\leq s\leq 4$, the only possible solutions are $s+t-2=0$ or $t-2=0$. Thus, $a=2$ is only possible for considered $s,t$ when $t=2$ or $t=2-s$, which correspond to triples of eigenvalues $(4+s,2,2-s)$ and $(4+s,2-s,2)$ with $0\leq s\leq 2$. We note that those triples of eigenvalues are attainable as we can choose any $c,d$ such that $0\leq c,d\leq 1$. 

Assume $(\Lambda_1,\Lambda_2,\Lambda_3)=(4+s,t,4-s-t)$ is not of the form $(4+s,2,2-s)$ and $(4+s,2-s,2)$ with $0\leq s\leq 2$. Then, if a solution of \eqref{cycle: inverse equation} exists then $a\neq 2$. Then, solving the system of the first two equations in \eqref{cycle: inverse equation}, we obtain that

\begin{align*}
4c&=\frac{a(16-s^2+t(4-s)-t^2-3a(4-a))-((16-s^2)t-(4+s)t^2)}{2(a-2)}\\
&=\frac{3a^3-12a^2+a(16-s^2+4t-st-t^2)+(s^2t+st^2+4t^2-16t)}{2(a-2)} \qquad \text{and}\qquad \\
4d&=\frac{(a-4)(16-s^2+t(4-s)-t^2-3a(4-a))+(16-s^2)t-(4+s)t^2}{2(a-2)}\\
&=\frac{3a^3-24a^2+a(64-s^2+4t-st-t^2)+(4s^2+4st-s^2t-st^2-64)}{2(a-2)}.
\end{align*}

Thus, we need to show for all $0\leq s\leq 4$, $t\geq0$, $t\neq 2$ and $t\neq 2-s$ such that $t+s\leq 4$, there exists $a\in[0,4]$ such that $0\leq 4c\leq a^2$ and $0\leq 4d\leq (4-a)^2$.

First, we show that for $2t+s\leq 4$, we can take $a=4-t\in[0,4]$. We have $4d=t^2$ and 
\begin{align*}
4c&=\frac{2s^2 t-4s^2+2st^2 -4st-2t^3+20t^2-64t+64}{2(2-t)}\\&=-(2t+s-4)^2-(8-3t)(2t+s-4)+t(4-t)\in[0,(4-t)^2]=[0,a^2].
\end{align*}

Since $(4-a)^2=t^2$, we have $(4d)\in[0,(4-a)^2]$. We show that $4c\in[0,a^2]=[0,(4-t)^2]$. We note that, by the conditions on $s$ and $t$, we have $2t-4\leq(2t+s-4)\leq 0$ and $0\leq t\leq 2$. In particular, $-\frac{8-3t}{2}\leq 2t-4$ so a function $-u^2-(8-3t)u+t(4-t)$ for $2t-4\leq u\leq 0$ and $0\leq t\leq 2$ has a minimum at $u=0$ which is $t(4-t)\geq 0$ and a maximum at $u=(2t-4)$ which is $(4-t)^2$.

Now, we show that for $2t+s>4$, we can take $a=s+t\in[0,4]$. We have
\begin{align*}
4d&=\frac{2(s+t-2)(s+t-4)^2}{2(s+t-2)}=(s+t-4)^2\in[0,(4-a)^2]\qquad \text{and}\\
4c&=\frac{2(s+t-2)(s^2+3st-4s+t^2)}{2(s+t-2)}=s^2+3st-4s+t^2\\
&=(2t+s-4)^2+(4-t)(2t+s-4)+4t-t^2\in[0,(s+t)^2]=[0,a^2].
\end{align*}
Since $(4-a)^2=(s+t-4)^2$, $4d\in[0,(4-a)^2]$. We show that $4c\in[0,a^2]=[0,(s+t)^2]$. We note that, by the conditions on $s$ and $t$, we have $\max\{0,2t-4\}\leq 2t+s-4\leq t$ and $0\leq t\leq 4$. Thus, we have $4c\geq 0$ and 
$$4c-(s+t)^2=(4-t)\Big((2t-4)-(2t+s-4)\Big)\leq 0.$$
\end{proof}

Combining Lemmas ~\ref{lemma: cycle all smaller than 4} and ~\ref{lemma: cycle one larger than 4} and rewriting the result for any normalization of the sum of eigenvalues, we obtain Theorem~\ref{theorem_cycle}.
\end{proof}

\section{Path}\label{sec: path graph}

Consider the path graph on four vertices, i.e., let $v_1,v_2,v_3,v_4$ be the vertices, then the set of edges consists of $v_iv_{i+1}$, $i=1,2,3$. To describe the realizable eigenvalues on the path graph, we treat it as a cycle graph with one edge having weight zero.

\begin{theorem}\label{theorem_path}
   The triple of eigenvalues obtained by the path graph on four vertices are precisely $(x,y,z)$ such that $x,y,z\geq 0$, $\max\{x,y,z\}\geq\frac{1}{2}(x+y+z)$, and 
\[\Big(9(x^3+y^3+z^3)-13U+62xyz\Big)^2\leq (3S-2P)^2(9S-14P),\]

where $S = x^2+y^2+z^2, \qquad P =xy+xz+yz,\qquad U = x^2(y+z) + y^2(x+z)+z^2(x+y)$.
\end{theorem}

\begin{proof} 
We view the path graph as a cycle with the edge $v_4v_1$ having weight $0$. In the notations of the proof of Theorem~\ref{theorem_cycle}, we have $l_1=0$ (the weight of the edge $v_4v_1$), and $l_i$ is the weight of the edge $v_{i-1}v_i$ where $i=2,3,4$. If the sum of eigenvalues is zero, then necessarily $x=y=z=0$ and the theorem holds, so we set aside this case. We adopt the normalization that the sum of eigenvalues is $2f>0$. By Lemma~\ref{lemma: cycle all smaller than 4} and the normalization that the sum of eigenvalues is $2f$, we know that at least one eigenvalue should be at least $f$. Consider a given triple $(\Lambda_1,\Lambda_2,\Lambda_3)$ of non-negative numbers summing to $2f$. Without loss of generality, we assume $\Lambda_1=f+s\geq f$, i.e., $0\leq s\leq f$. Denote $\Lambda_2=\frac{f-s-t}{2}$ and $\Lambda_3=\frac{f-s+t}{2}$ where $|t|\leq f-s$. To describe which triples $(\Lambda_1,\Lambda_2,\Lambda_3)$ are realizable by a path graph, we need to describe pairs $(s,t)$ with $0\leq s\leq f$ and $|t|\leq f-s$ so that there exist $a$ and $d$ such that
\begin{align}
3a(f-a)+4d&=(f+s)(f-s)+\frac{(f-s)^2-t^2}{4}\nonumber\\
4ad&=(f+s)\cdot\frac{(f-s)^2-t^2}{4}\label{path: eigenvalues inverse}\\
0\leq &a\leq f\nonumber\\
0\leq &d\leq\frac{(f-a)^2}{4}.\nonumber
\end{align}

For the relation between the eigenvalues, the weights of the graph edges, and the parameters $a,d$, see \eqref{cycle: parameters vs lengths}, \eqref{cycle: eigenvalues vs parameters}, and \eqref{cycle: inverse equation}.

From the first two equations in \eqref{path: eigenvalues inverse}, we obtain that $a$ must satisfy the following equation.

\begin{equation}\label{path: equation on a}
12a^3-12fa^2+\Big(4(f^2-s^2)+(f-s)^2-t^2\Big)a-(f+s)\Big((f-s)^2-t^2\Big)=0
\end{equation}
In particular, we note that $a=0$ if and only if $t=\pm(f-s)$, i.e., $(\Lambda_1,\Lambda_2,\Lambda_3)=(f+s,0,f-s)$ or $(f+s,f-s,0)$. Thus, eigenvalues $(f+s,0,f-s)$ and $(f+s,f-s,0)$ where $0\leq s\leq f$ are realizable if $a=0$ and $d=\frac{f^2-s^2}{4}$. Moreover, if $a\neq 0$, then from the second equation, conditions on $s,t$, and $a$, we have that $d$ is automatically non-negative.

Assume $(\Lambda_1,\Lambda_2,\Lambda_3)\neq(f+s,0,f-s)$ or $(f+s,f-s,0)$ where $0\leq s\leq f$. Then, $a\neq 0$ and it is enough to describe $(s,t)$ with $0\leq s\leq f$ and $|t|< f-s$ so that there exists $a\in(0,f]$ such that \eqref{path: equation on a} holds and
\begin{equation}\label{path: inequality on d in a terms}
f^2-s^2+\frac{(f-s)^2-t^2}{4}-3a(f-a)\leq (f-a)^2.
\end{equation}
First, we analyze the inequality \eqref{path: inequality on d in a terms}. It can be rewritten as
\[
a^2-\frac{f}{2}a+\frac{-3s^2-2sf+f^2-t^2}{8}\leq 0
\]
that has a solution if an only if
\begin{equation}\label{path: 1st cond on s,t}
    3\Big(s+\frac{f}{3}\Big)^2+t^2\geq\frac{5}{6}f^2
\end{equation}
and the solution is
\begin{equation}\label{path: cond on a from d condition}
   \frac{f}{4}-\frac{1}{4}\sqrt{6s^2+4sf+2t^2-f^2}\leq a\leq \frac{f}{4}+ \frac{1}{4}\sqrt{6s^2+4sf+2t^2-f^2}
\end{equation}

We note that $\frac{f}{4}+\frac{1}{4}\sqrt{6s^2+4sf+2t^2-f^2}\leq f$ as 
\begin{align*}
6s^2+4sf+2t^2-f^2\leq 6s^2+4sf+2(f-s)^2-f^2=8s^2+f^2\leq 8\cdot f^2+f^2=9f^2 
\end{align*}
and so \begin{align}\label{path: some inequality}
    \frac{f}{4}+\frac{1}{4}\sqrt{6s^2+4sf+2t^2-f^2}\leq \frac{f}{4}+\frac{1}{4}\sqrt{9f^2}=f
\end{align}

Thus, we are looking for an existence of a solution $a>0$ of \eqref{path: equation on a} such that \eqref{path: cond on a from d condition} holds.

Denote
\[
F(a)=12a^3-12fa^2+\Big(4(f^2-s^2)+(f-s)^2-t^2\Big)a-(f+s)\Big((f-s)^2-t^2\Big).
\]
Then,
\begin{equation}\label{path: monotonicity regions w.r.t. a}
F'(a)=36 a^2 - 24 a f + 5f^2-2sf-3s^2 - t^2
\end{equation}
Moreover, if we try to solve $F'(a)=0$ with respect to $a$, we get the discriminant 
\begin{equation}\label{path: discriminant expression}
D=144\Bigg(3\Big(s+\frac{f}{3}\Big)^2+t^2-\frac{4}{3}f^2\Bigg).
\end{equation}

\textbf{{Case 1:  $\frac{f}{4}-\frac{1}{4}\sqrt{6s^2+4sf+2t^2-f^2}\leq 0$}}, i.e., $3\Big(s+\frac{f}{3}\Big)^2+t^2\geq \frac{4}{3}f^2$. In particular, $D\geq 0$. Denote $L=\sqrt{3s^2+2fs+t^2-f^2}$. Then, $D=144L^2$ and $\frac{f}{4}\pm\frac{1}{4}\sqrt{6s^2+4sf+2t^2-f^2}= \frac{f\pm\sqrt{2L^2+f^2}}{4}$. We have $F'(a)=0$ if and only if $a=\frac{2f\pm L}{6}$. Also, $F'(a)<0$ for $a\in\Big(\frac{2f-L}{6},\frac{2f+L}{6}\Big)$ so $F$ is decreasing for those $a$, and $F'(a)>0$ for $a\in\Big(-\infty,\frac{2f-L}{6}\Big)\cup\Big(\frac{2f+L}{6},+\infty\Big)$. By \eqref{path: some inequality}, we know that $L\leq 2f$ so $\frac{2f-L}{6}\geq 0$. Moreover, $F(0)=-(f+s)\Big((f-s)^2-t^2\Big)<0$ as $0\leq s\leq f$ and $|t|<f-s$ so $F\Big(\frac{f-\sqrt{2L^2+f^2}}{4}\Big)<0$ as well. Also, $\frac{f+\sqrt{2L^2+f^2}}{4}\geq \frac{2f+L}{6}$ as $L,f\geq 0$. 
Thus, there exists $a>0$ such that \eqref{path: cond on a from d condition} holds and $F(a)=0$ if and only if 
\[
F\Big(\frac{f+\sqrt{2L^2+f^2}}{4}\Big)\geq 0\qquad \text{or} \qquad F\Big(\frac{2f-L}{6}\Big)\geq 0.
\]
We note that $4(f^2-s^2)+(f-s)^2-t^2=4f^2-L^2$. Then,
\[
F\Big(\frac{f+\sqrt{2L^2+f^2}}{4}\Big)-F\Big(\frac{2f-L}{6}\Big)=-\frac{1}{72}(14f^3+21fL^2+8L^3)+\frac{1}{8}(2f^2+L^2)\sqrt{2L^2+f^2}.
\]
Moreover,
\[
9^2(2f^2+L^2)^2(2L^2+f^2)-(14f^3+21fL^2+8L^3)^2=2(2f-L)^2(7L^2+2Lf+4f^2)^2\geq 0.
\]
Thus, $F\Big(\frac{f+\sqrt{2L^2+f^2}}{4}\Big)-F\Big(\frac{2f-L}{6}\Big)\geq 0$ so there exists $a>0$ such that \eqref{path: cond on a from d condition} holds and $F(a)=0$ if and only if $F\Big(\frac{f+\sqrt{2L^2+f^2}}{4}\Big)\geq 0$.

\textbf{{Case 2: $\frac{f}{4}- \frac{1}{4}\sqrt{6s^2+4sf+2t^2-f^2}>0$}}, i.e., $\frac{5}{6}f^2\leq3\Big(s+\frac{f}{3}\Big)^2+t^2< \frac{4}{3}f^2$ as \eqref{path: 1st cond on s,t} should hold. In particular, $D<0$ so there are no critical points. Also, $F'(0)>0$, so the function $F$ is increasing. Thus, there exists a solution $a>0$ of \eqref{path: equation on a} such that \eqref{path: cond on a from d condition} holds if and only if $F\Big(\frac{f}{4}- \frac{1}{4}\sqrt{6s^2+4sf+2t^2-f^2}\Big)\leq 0$ and $F\Big(\frac{f}{4}+ \frac{1}{4}\sqrt{6s^2+4sf+2t^2-f^2}\Big)\geq 0$.

We note that since $F\Big(\frac{f}{4}- \frac{1}{4}\sqrt{6s^2+4sf+2t^2-f^2}\Big)\leq 0$ in Case 1, we can combine Cases 1 and 2. Now, we plug in our values of $a$ into $F(a)$ to obtain
\begin{align*}
    &F\Bigg(\frac{f}{4}\pm \frac{1}{4}\sqrt{6s^2+4sf+2t^2-f^2}\Bigg) \\&= -\frac{1}{8}(f^3+2f^2s+7fs^2-3ft^2+8s^3-8st^2\mp(f^2+2fs+3s^2+t^2)\sqrt{6s^2+4fs+2t^2-f^2}).
\end{align*}
As a result, changing variables to $x = \Lambda_1=f+s, y=\Lambda_2 = \frac{1}{2}(f-s-t), z=\Lambda_3 = \frac{1}{2}(f-s+t)$ (note $x,y,z\geq 0$) and introducing notation \begin{align*}
    S = x^2+y^2+z^2 \qquad P =xy+xz+yz\qquad U = x^2(y+z) + y^2(x+z)+z^2(x+y),
\end{align*}
we obtain that there exists a solution $a>0$ of \eqref{path: equation on a} such that \eqref{path: cond on a from d condition} holds if and only if
\[|9(x^3+y^3+z^3)-13U+62xyz|\leq (3S-2P)\sqrt{9S-14P},\]
which can be equivalently rewritten as
\begin{align}
    \Big(9(x^3+y^3+z^3)-13U+62xyz\Big)^2\leq (3S-2P)^2(9S-14P)\label{path - final result}
\end{align}
because $3S-2P=\frac{3}{2}\Bigg((x-y)^2+(x-z)^2+(y-z)^2\Bigg)+P\geq 0$ for $x,y,z\geq 0$.

Finally, note that if \eqref{path - final result} holds, the right-hand side must necessarily be non-negative, meaning $9S-14P\geq 0$. Translating back into $f, s, t$, this exactly gives \eqref{path: 1st cond on s,t}, so that condition has been satisfied as well. 
\end{proof}

\section{Kite}\label{sec: kite graph}

Consider a graph on four vertices $v_1,v_2,v_3,v_4$ whose set of edges consists of $v_1v_2$, $v_2v_3$, $v_1v_3$, and $v_3v_4$. We will call it a kite graph (as, for example, in \cite{TOPCU2018134}). In the literature, this graph is also called a paw graph (see, for example, \cite{FGL}). Similarly to the previous graphs, we choose a convenient parametrization.

\begin{theorem}\label{theorem_kite}
     The triple of eigenvalues obtained by the kite graph on four vertices are precisely $(x,y,z)$ such that $x,y,z\geq0$ and at least one of the following inequalities holds:
    \begin{itemize}
        \item $\sqrt{3}\vert x-y\vert\geq x+y$,
        \item $\sqrt{3}\vert x-z\vert\geq x+z$,
        \item $\sqrt{3}\vert y-z\vert\geq y+z$.
    \end{itemize}
\end{theorem}
\begin{proof}

 Let $\{l_1, l_2, l_3, l_4\}$ be the weights of edges $v_1v_2$, $v_2v_3$, $v_1v_3$, $v_3v_4$, respectively, in the considered cycle graph on fours vertices. Then, the weighted graph Laplacian has the form

\begin{align}\label{kite_matrix}
\begin{bmatrix}
l_1+l_3 & -l_1 & -l_3 & 0 \\
-l_1 & l_1+l_2 & -l_2 & 0 \\
-l_3 & -l_2 & l_2+l_3+l_4 & -l_4 \\
0 & 0 & -l_4 & l_4 \\
\end{bmatrix}.
\end{align}

Calculating the characteristic equation, we get that, excluding the necessarily $0$ eigenvalue, the rest of eigenvalues must satisfy
\begin{align}\label{kite_equation}
    \lambda^3 - 2 \lambda^2 (l_1+l_2+l_3+l_4) + 3 \lambda (l_1 l_2+l_1l_3+l_2l_3+l_2l_4+l_3l_4) + 4 \lambda (l_1 l_4) \\- 4 (l_1 l_2 l_4 + l_1 l_3 l_4 + l_2 l_3 l_4)=0.\nonumber
\end{align}
Thus, the roots $\lambda_1, \lambda_2, \lambda_3$ of this cubic equation satisfy the following system of equations.

\begin{align}\label{kite: equations for symmetric polynomials in roots}
    \lambda_1+\lambda_2+\lambda_3 & = 2(l_1+l_2+l_3+l_4), \nonumber\\
    \lambda_1\lambda_2 + \lambda_2\lambda_3 + \lambda_1\lambda_3 & = 3(l_1 l_2+l_1l_3+l_2l_3+l_2l_4+l_3l_4) + 4l_1l_4, \\
    \lambda_1\lambda_2\lambda_3 & = 4 (l_1 l_2 l_4 + l_1 l_3 l_4 + l_2 l_3 l_4).\nonumber
\end{align}

Let
\begin{equation}\label{kite: parametrization}
a=l_1l_4, \qquad b=l_2+l_3, \qquad c=l_1l_2+l_1l_3+l_2l_3, \quad\text{and}\quad d=l_4.
\end{equation}

We again apply the normalization that $l_1+l_2+l_3+l_4=4$. Thus, we obtain $l_1+l_4=4-b$ and $a=l_1l_4=(l_1+l_4-l_4)l_4=(4-b-d)d$. We also notice $l_2l_3=c-l_1(l_2+l_3)=c-b(4-b-d)$. In order to make sure we can recover positive real weights $l_1,l_2,l_3,l_4$ from $a,b,c,d$, it is necessary and sufficient to have the following inequalities hold.

\begin{equation}\label{kite: conditions}
 0\leq a,b,d\leq 4, \quad 4a\leq(4-b)^2, \quad 4(c-b(4-b-d))\leq b^2,\quad c\geq 0,\quad b(4-b-d)\leq c\leq\frac{16}{3}.
\end{equation}

We simplify the above system of inequalities.

\begin{proposition}\label{kite: prop simplified conditions}
  The system of inequalities in \eqref{kite: conditions} is equivalent to 
  \begin{equation*}
b,d\geq 0, \quad b+d\leq 4, \quad  b(4-b-d)\leq c\leq 4b-bd-\frac{3}{4}b^2.
  \end{equation*}
\end{proposition}

\begin{proof}
First, we show that for $b,d\geq 0$, we have $b,d\leq 4$ and $0\leq a\leq 4$ if and only if $b+d\leq 4$. We recall that $a=\Big(4-(b+d)\Big)d$. If $a=0$ then $b+d=4$ or $d=0$ so $b+d=b\leq 4$. Otherwise, if $a>0$ then $d>0$ so $4-(b+d)>0$, i.e., $b+d\leq 4$. On the other hand, if $b,d\geq 0$ and $b+d\leq 4$ then $b,d\leq 4$, $a\geq 0$, and $a\leq (4-d)d\leq 4$. 

If $b+d\leq 4$ and $b\geq 0$, then $b(4-b-d)\geq 0$ so $c\geq 0$. We also note that $4b-bd-\frac{3}{4}b^2-b(4-b-d)=\frac{1}{4}b^2\geq 0$ so the system of inequalities in the statement is consistent.

Moreover, 
\begin{equation*}
(4-b)^2-4a=(4-b)^2-4(4-b-d)d=(b+2d-4)^2\geq 0
\end{equation*}
so $4a\leq (4-b)^2$ for any $b,d$.

Finally,
\begin{equation*}
4(c-b(4-b-d))\leq b^2 \Leftrightarrow c\leq -\frac{3}{4}b^2+b(4-d),
\end{equation*}
where the quadratic polynomial in $b$ is maximized at $b = \frac{2}{3}(4-d)$, so 
\[c\leq -\frac{3}{4}\left(\frac{2}{3}(4-d)\right)^2+\frac{2}{3}(4-d)(4-d)=\frac{1}{3}(4-d)^2\leq \frac{16}{3}\quad\text{as}\quad 0\leq d\leq 4.
\]
\end{proof}

Moreover, using our parametrization, \eqref{kite_equation} can be rewritten as
\begin{equation}\label{kite: cubic equation in terms of parameters}
\lambda^3 - 8 \lambda^2  + \lambda (3c+3bd+4d(4-b-d)) - 4 cd=0
\end{equation}

Consider a given triple $(\Lambda_1,\Lambda_2,\Lambda_3)$ of non-negative numbers summing to $8$. Without loss of generality assume $\Lambda_1\geq\Lambda_2\geq\Lambda_3$. Let $\Lambda_1=\Lambda$. Then, factoring out $(\lambda-\Lambda)$ in \eqref{kite: cubic equation in terms of parameters}, we obtain that the other two eigenvalues are roots of $\lambda^2+(\Lambda-8)\lambda+\Lambda(\Lambda-8)+3c+3bd+4d(4-b-d)=0$ so using $\mathcal{D}$ for the discriminant of this quadratic, 
\begin{equation}\label{kite: L2,3 equation}
    \Lambda_{2,3} =\frac{8-\Lambda\pm\sqrt{\mathcal{D}}}{2}= \frac{8-\Lambda\pm\sqrt{(\Lambda-8)^2-4(\Lambda(\Lambda-8)+3c+3bd+4d(4-b-d))}}{2}.
\end{equation}

\begin{lemma}\label{kite: lemma which are not attainable}
Assume $(\Lambda_1,\Lambda_2,\Lambda_3)$ above does not satisfy any conditions in the statement of Theorem~\ref{theorem_kite}. Then, $(\Lambda_1,\Lambda_2,\Lambda_3)$ is not attainable by a kite graph.
\end{lemma}
\begin{proof}
Since $\Lambda_1\geq\Lambda_2\geq\Lambda_3\geq0$, the fact that $(\Lambda_1,\Lambda_2,\Lambda_3)$ does not satisfy any conditions in the statement of Theorem~\ref{theorem_kite} is equivalent to 
\begin{equation*}
\sqrt{3}(\Lambda_1-\Lambda_3)<\Lambda_1+\Lambda_3, \quad\text{i.e.,} \quad\Lambda_1<(2+\sqrt{3})\Lambda_3.
\end{equation*}
Using $\Lambda_1=\Lambda$ and the expressions \eqref{kite: L2,3 equation}
 of $\Lambda_{2,3}$, we show that it is impossible to have simultaneously $\Lambda_1\geq\Lambda_2$ and $\Lambda_1<(2+\sqrt 3)\Lambda_3$. Having both of those inequalities is equivalent to having $0\leq \Lambda\leq 8$ such that 
 \begin{equation}\label{kite: bounds on root of D}
3\Lambda-8\geq \sqrt{\mathcal D}\quad\text{and}\quad 8-(5-\sqrt{3})\Lambda> \sqrt{\mathcal D}.
 \end{equation}
We note that $3\Lambda-8\geq 8-(5-\sqrt 3)\Lambda$ if and only if $\Lambda\geq \frac{16}{8-\sqrt 3}$. Moreover, for $\Lambda\geq \frac{16}{8-\sqrt 3}$, we have $8-(5-\sqrt 3)\Lambda\leq -\frac{8}{8-\sqrt 3}(2-\sqrt 3)<0$, and, for $0\leq \Lambda<\frac{16}{8-\sqrt 3}$, we have $3\Lambda-8<-\frac{8}{8-\sqrt 3}(2-\sqrt 3)<0$. Thus, \eqref{kite: bounds on root of D} is not possible, and we have proved the lemma.
\end{proof}

\begin{lemma}\label{kite: lemma which are realizable}
Consider $(\Lambda_1,\Lambda_2,\Lambda_3)$ with $\Lambda_1\geq\Lambda_2\geq\Lambda_3\geq 0$ and $\Lambda_1+\Lambda_2+\Lambda_3=8$. Assume that $\sqrt{3}(\Lambda_1-\Lambda_3)\geq \Lambda_1+\Lambda_3$ (i.e., $\Lambda_1\geq (2+\sqrt 3)\Lambda_3$). Then, $(\Lambda_1,\Lambda_2,\Lambda_3)$ is attainable by a kite graph.
\end{lemma}
\begin{proof}
Let $(\Lambda_1, \Lambda_2, \Lambda_3)$ be as in the statement of the lemma. We show that there exists $b,c,$ and $d$ satisfying the conditions in Proposition~\ref{kite: prop simplified conditions} so that $\Lambda_1,\Lambda_2,\Lambda_3$ are the roots of \eqref{kite: cubic equation in terms of parameters}. In particular, we can take $c=4b-bd-\frac{3}{4}b^2$. 
Then, we have
\begin{align*}
3c+3bd+4d(4-b-d) & = 3(4b-bd-\frac{3}{4}b^2)+3bd+4d(4-b-d) \\&= \frac{-9b^2}{4} + 12b-4bd+16d-4d^2\\
4cd & = 4(4b-bd-\frac{3}{4}b^2)d = 16bd - 4bd^2-3b^2d
\end{align*}
so the roots of the cubic equation \eqref{kite: cubic equation in terms of parameters} for our choice of $c$ are
\begin{equation*}
\frac{1}{2}\Big(-3b-4d+16\Big), \qquad \frac{3}{4}b+d\pm\frac{1}{4}\sqrt{9b^2-8bd+16d^2}.
\end{equation*}

We show that there exist $b,d\geq 0$ such that $b+d\leq 4$ so that
\begin{align*}
\Lambda_1&=\frac{3}{4}b+d+\frac{1}{4}\sqrt{9b^2-8bd+16d^2},\\
\Lambda_2&=\frac{1}{2}\Big(-3b-4d+16\Big),\\
\Lambda_3&=\frac{3}{4}b+d-\frac{1}{4}\sqrt{9b^2-8bd+16d^2}.
\end{align*}
Since $\Lambda_1+\Lambda_2+\Lambda_3=8$, we just need to show that there exist $b,d\geq 0$ such that $b+d\leq 4$ so that
\begin{align}\label{kite: equation on b and d}
&2(\Lambda_1-\Lambda_3)=\sqrt{9b^2-8bd+16d^2},\\
&2(\Lambda_1+\Lambda_3)=3b+4d.
\end{align}
We note that for any $b,d\geq 0$ we have $3(9b^2-8bd+16d^2)\geq (3b+4d)^2$ as it is equivalent to $2(3b-4d)^2\geq 0$, which aligns with our condition that $\sqrt{3}(\Lambda_1-\Lambda_3)\geq \Lambda_1+\Lambda_3$. Moreover, we have \[2(3b-4d)^2=3\Big(2(\Lambda_1-\Lambda_3)\Big)^2-\Big(2(\Lambda_1+\Lambda_3)\Big)^2.\]
Thus, there are two cases
\begin{align*}
&(b,d)
=\Bigg(\frac{(\Lambda_1+\Lambda_3)\pm\sqrt{(\Lambda_1-\Lambda_3)^2-2\Lambda_1\Lambda_3}}{3},\frac{(\Lambda_1+\Lambda_3)\mp\sqrt{(\Lambda_1-\Lambda_3)^2-2\Lambda_1\Lambda_3}}{4}\Bigg).
\end{align*}
In particular, since $\Lambda_1\geq \Lambda_3\geq 0$, $b,d\geq 0$. We also show that
\begin{equation*}
(b,d)=\Bigg(\frac{(\Lambda_1+\Lambda_3)-\sqrt{(\Lambda_1-\Lambda_3)^2-2\Lambda_1\Lambda_3}}{3},\frac{(\Lambda_1+\Lambda_3)+\sqrt{(\Lambda_1-\Lambda_3)^2-2\Lambda_1\Lambda_3}}{4}\Bigg)
\end{equation*}
satisfies $b+d\leq 4$. Recall that, since $\Lambda_1+\Lambda_2+\Lambda_3=8$ and $\Lambda_1\geq\Lambda_2\geq\Lambda_3\geq 0$, we have that $\Lambda_1+2\Lambda_3\leq 8$. Thus,
\begin{equation}\label{kite: equality on b+d}
12(b+d)=7(\Lambda_1+\Lambda_3)-\sqrt{(\Lambda_1-\Lambda_3)^2-2\Lambda_1\Lambda_3}.
\end{equation}
In particular, if $\Lambda_1+\Lambda_3\leq \frac{48}{7}$, then $b+d\leq 4$. Assume now that $\Lambda_1+\Lambda_3>\frac{48}{7}$ so $0\leq\Lambda_3<\frac{8}{7}$. 
Let 
\begin{align*}
f(\Lambda_1,\Lambda_3)&=\Bigg(7(\Lambda_1+\Lambda_3)-48\Bigg)^2-\Bigg((\Lambda_1-\Lambda_3)^2-2\Lambda_1\Lambda_3\Bigg)\\
&=6\Bigg(8\Lambda_1^2-(112-17\Lambda_3)\Lambda_1+8(\Lambda_3^2-14\Lambda_3+48)\Bigg).
\end{align*}
We have that for fixed $\Lambda_3$, $f(\Lambda_1,\Lambda_3)$  decreases if $\Lambda_1\leq\frac{112-17\Lambda_3}{16}$ and increases if $\Lambda_1\geq\frac{112-17\Lambda_3}{16}$. Also, we observe that $\frac{112-17\Lambda_3}{16}>\frac{48}{7}-\Lambda_3$. Note that $f(8-2\Lambda_3,\Lambda_3) = 12\Lambda_3(3\Lambda_3-4)<0$ for $\Lambda_3<\frac{8}{7}<\frac{4}{3}$. Moreover,
\[
f\Big(\frac{48}{7}-\Lambda_3,\Lambda_3\Big)=-6\Bigg(\Lambda_3^2-\frac{48}{7}\Lambda_3+\frac{384}{49}\Bigg)\leq -6\Bigg(\left(\frac{8}{7}\right)^2-\frac{48}{7}\cdot\frac{8}{7}+\frac{384}{49}\Bigg)=-6\cdot\frac{64}{49}<0.
\]
Thus, $f(\Lambda_1,\Lambda_3)<0$ for $8-2\Lambda_3\geq \Lambda_1>\frac{48}{7}-\Lambda_3$ and $\Lambda_3\geq 0$. Therefore, using \eqref{kite: equality on b+d}, we have $12(b+d)\leq 48$, i.e., $b+d\leq 4$. 

As a result, we found $b,c$, and $d$ so that the conditions in Proposition~\ref{kite: prop simplified conditions} hold the lemma holds.
\end{proof}

Combining Lemmas~\ref{kite: lemma which are not attainable} and \ref{kite: lemma which are realizable} and considering any possible order of eigenvalues and normalization, we obtain Theorem~\ref{theorem_kite}.
\end{proof}

\section{Complete graphs}\label{sec: complete graph}
The complete graph on $n$ vertices, $K_n$, is a graph with any two vertices connected by an edge. In this section, using the argument in \cite{CdV88}, we show that

\begin{theorem}\label{theorem_kn}
    For any $n$-tuple $\bm\Lambda=(\Lambda_1, \dots, \Lambda_n)$ of $n$ nonnegative numbers, there exists a weight function $w$ on the edges of the complete graph $K_{n+1}$ on $n+1$ vertices such that $(K_{n+1}, w)$ has eigenvalues exactly corresponding to $\bm\Lambda$.
\end{theorem}
To find the eigenvalues attained by complete graphs on $n$ vertices, we introduce the concept of graph suspensions. 
\begin{definition}
    The suspension of a graph $G$ with vertices $V(G)=\{v_1, v_2, \dots, v_n\}$ is the graph $G'$ with $V(G')=V(G)\cup\{v_{n+1}\}$ and $E(G')=E(G)\cup\{v_iv_{n+1}: 1\leq i\leq n\}$. 
\end{definition}

In \cite[\S 4]{CdV88}, Colin de Verdière described the eigenvalues of the suspension of the graph where weights are assigned not only to edges but the vertices as well and equal weights are assigned to the added edges in the suspension. In Theorem~\ref{theorem_kn}, we do not have various weights on vertices so we provide Colin de Verdière's argument here for our situation for completeness as the argument is simpler for our case.

\begin{lemma}(compare with \cite[\S 4]{CdV88})\label{complete graph: lemma suspension}
    If a weighted graph $G$ attains the nonnegative eigenvalues $(\lambda_1, \dots, \lambda_n)$, then its suspension with the added edges having weight $c\geq 0$ will attain the eigenvalues \[(\lambda_1+c, \lambda_2+c, \dots, \lambda_n+c, c(n+1)).\]
\end{lemma}
\begin{proof}
Let $B_n$ be an $n\times n$ weighted Laplacian matrix for the graph $G$. Recall that the Laplacian has the form (degree)-(adjacency matrix), so the Laplacian for the suspension $G'$ obtained by adding edges of weight $c$ will have the form 
$$B' = \begin{bmatrix}
    & & &-c\\
    & B_n+cId_n& & \dots\\
    & && -c&\\
    -c & \dots & -c& cn
\end{bmatrix}$$
Suppose $\bm v=\langle
    v_1, v_2, \dots, v_n
\rangle$ is an eigenvector of $B_n$ with eigenvalue $\lambda_{\bm v}$. Recall that the vector $\textbf{1}_n$ with $n$ entries being all $1$ is an eigenvector of $B_n$ with eigenvalue $0$. Then, since $B_n$ is symmetric, then either $\lambda_{\bm v}\neq 0$ so $\bm v$ and $\textbf{1}_n$ are orthogonal or we can assume that $\bm v$ is orthogonal to $\textbf{1}_n$. Consider $\bm v'=\langle v_1, v_2, \dots, v_n, 0\rangle$. We show that $\bm v'$ is an eigenvector of $B'$.

Indeed, 

$$B'\bm v'=\begin{bmatrix}
    & & &-c\\
    & B_n+cId_n& & \dots\\
    & && -c&\\
    -c & \dots & -c& cn
\end{bmatrix}\begin{bmatrix}
    v_1\\v_2\\\dots\\v_n\\0
\end{bmatrix} 
=\begin{bmatrix}
    \lambda_{\bm v} v_1+cv_1\\
    \lambda_{\bm v} v_2 + cv_2\\
    \dots\\
    \lambda_{\bm v} v_n + cv_n\\
    -c\cdot \sum_{i=1}^n v_i
\end{bmatrix}$$

We note that $\sum_{i=1}^n v_i = \bm v\cdot \textbf{1}_n=0$ as $\bm v$ and $\textbf{1}_n$ are orthogonal. Thus, we get $B'\bm v'=(\lambda_{\bm v}+c)\bm v'$, so $\bm v'$ is an eigenvector of $B'$ with eigenvalue $\lambda_{\bm v}+c$. Repeating this process with the other eigenvectors of $B_n$ that have nonzero eigenvalues, we will obtain eigenvalues $\lambda_1+c, \dots, \lambda_n+c$. Moreover, if $c=0$, then $\langle0,\ldots, 0, 1\rangle$ is an eigenvector , and if $c>0$, we will obtain a new nonzero eigenvalue. 

Assume $c>0$. 
Consider a vector $\bm w=\langle1,\dots,1, y\rangle$ for some $y$. We show that there exist $y\neq 0$ and $\mu$ such that $B'\bm w=\mu \bm w$. Expanding $B'\bm w$, we get $$\langle -c(y-1), -c(y-1),\dots,-c(y-1), nc(y-1)\rangle$$ and we want this to be equal to $\langle\mu, \mu, \dots, \mu, \mu y\rangle$. Solving, we have $\mu=-c(y-1)$ and $y=-n$ or $y=1$. We exclude $y=1$ as that is the constant vector case leading to the zero eigenvalue, so we can choose $y=-n$ and $\mu=c(n+1)$.

Thus, we obtain the lemma.
\end{proof}

We now use this lemma to prove Theorem \ref{theorem_kn}.
\begin{proof}
    We prove this result by induction. The base case, that $K_2$ attains any nonnegative eigenvalue. The $2\times2$ matrix $\begin{bmatrix}\frac{\lambda_1}{2}&-\frac{\lambda_1}{2}\\-\frac{\lambda_1}{2}&\frac{\lambda_1}{2}\end{bmatrix}$ has eigenvalues $\lambda_0=0$ and $\lambda_1$ (see also \cite[Observation 2.1]{FGL} ). We now argue the inductive step. 

Suppose that any list of $n-1$ nonnegative numbers is attainable by a weighted $K_n$ graph. We recall that we allow zero weights for edges. Let $(\lambda_1,\lambda_2,\ldots,\lambda_{n-1},\lambda_n)$ be a given list of $n$ nonnegative numbers that we want to realize as the eigenvalues of weighted $K_{n+1}$ graph. We order the list so that $\lambda_1\geq \lambda_2\geq\ldots\geq \lambda_n$. We consider $K_n$ as a suspension of $K_{n-1}$.

Assume $\lambda_n=0$. Let $(K_{n-1},w_{0})$ be a weighted graph with eigenvalues $(\lambda_1,\ldots,\lambda_n)$. Then, we construct a weighted $K_n$ graph by assigning weights according to $w_0$ on a $K_{n-1}$ subgraph and $0$ to all other edges. By Lemma~\ref{complete graph: lemma suspension}, we obtain a weighted $K_n$ graph with desired eigenvalues.

If $\lambda_n\neq 0$, we choose $c$ such that $c(n+1)=\lambda_n$, i.e., $c=\lambda_n/(n+1)$. We note that $\lambda_i-c\geq 0$ for all $i\in\{1,2,\ldots, n-1\}$ as $\lambda_i\geq \lambda_n\geq 0$. By the inductive hypothesis, there exists a weight function $w$ on $K_{n}$ such that $(K_{n}, w)$ attains the eigenvalues $(\lambda_1-c, \dots, \lambda_{n-1}-c)$. We construct a weighted $K_{n+1}$ graph by assigning weights according to $w$ on a $K_{n}$ subgraph and $c$ to all other edges. By Lemma~\ref{complete graph: lemma suspension}, we obtain a weighted $K_{n+1}$ graph with desired eigenvalues.
\end{proof}

\section{Open questions}\label{sec: open questions}
In this section, we formulate several natural questions motivated by the results in this paper.

\begin{question}
What lists of $2g-3$ positive numbers in $[0,1/4]$ can be realized as the first 
$2g-3$ eigenvalues of the hyperbolic Laplacian 
on some Riemann surface of genus $g$?
\end{question}

For degenerating Riemann surfaces, that question 
is closely connected to the {\em Inverse Eigenvalue 
Laplace Problem} formulated in \cite{FGL}: 
\begin{question}
For a given connected 
graph $G$ 
on $n$ vertices, what lists of
$(n-1)$ nonnegative numbers can 
be realized as eigenvalues of a (weighted) 
graph Laplacian, for all possible nonegative weights 
assigned to the edges of $G$? 
\end{question}
 
We note that the complete graph $K_n$ is the only graph which realizes any list of $(n-1)$ nonnegative numbers as the eigenvalues of $K_n$ with some weight function \cite[\S4]{CdV88}, \cite[Thm. 2.5]{FGL}. Below we formulate a question and conjectures on how the set of realizable eigenvalues grow as we add an edge to a non-complete graph or consider its suspension.

Consider a graph $G$ on $n$ vertices. Let \[R=\{(\lambda_1,\lambda_2,\ldots,\lambda_{n-1})\,|\, \lambda_1+\lambda_2+\ldots+\lambda_{n-1}=1,\, \lambda_i\geq 0\,\text{for all}\, i\in\{1,\ldots, n-1\}\}.\]
We denote by $P(G)$ the ratio of the area of realizable $n-1$ eigenvalues of $G$ within $R$ to the area of $R$. In particular, by Theorem ~\ref{theorem_kn},  $P(K_n)=1$ for all $n\geq 2$.

\begin{question}
Consider a non-complete graph $G$. Let $G'$ be a graph obtained from $G$ by adding an edge. Is $P(G')>P(G)$?
\end{question}

By Theorem~\ref{thm: introthm}, $P(\text{cycle on $4$ vertices})>P(\text{path on $4$ vertices})$ and $P(\text{kite on $4$ vertices})>P(\text{star on $4$ vertices})$.

\begin{question}
Let $G$ be a non-complete graph and $G'$ be the suspension of $G$. Is $P(G')\geq P(G)$? 
\end{question}

We note that if $G$ is a path on 3 vertices, then $P(G)=1/2$ by \cite[Proposition 2.3]{FGL}. Its suspension $G'$ contains a cycle graph on $4$ vertices and $P($cycle graph on 4 vertices$)=3/4$ by Theorem~\ref{thm: introthm}. We thus have $P(G')>P(G)$ in this case.

\begin{question}
 Let $G_0$ be a non-complete graph. Consider a sequence of graphs $\{G_n\}_{\mathbb N}$ such that $G_{n+1}$ is the suspension of $G_n$ for all $n$. Is $\lim\limits_{n\rightarrow\infty}P(G_n)=1$?
\end{question}

\section*{Acknowledgments}
A. E. was supported by NSF grant DMS-2247230 and DMS-2552860. D. J. was supported by NSERC, FQRNT, and Peter Redpath fellowship offered by McGill University. A.T. was supported by NSERC and FRQNT. 
The authors would like to thank Y. Colin de Verdi\`ere, 
M.~Fortier Bourque, M. Karpukhin, and B. Petri
for interesting discussions. We thank the anonymous referee for useful comments, in particular, for communicating to us a proof of Lemma~\ref{lemma: cycle all smaller than 4} (provided in this paper) obtained by Zhan-Min Hsu (National Taiwan University) which is simpler than our original proof of this lemma.
\bibliographystyle{alpha}
\bibliography{bibliography}

\end{document}